\documentclass{svmult}
\usepackage{amsmath}
\usepackage{amssymb}
\usepackage{import} % Use import instead of input; for pdf_tex in figure folder

\usepackage{tikz-cd}

\newcommand{\concept}[1]{\textbf{#1}}

\newcommand{\ct}{\setminus}
\newcommand{\R}{\mathbb{R}}
\newcommand{\sph}{\mathbb{S}}

\newcommand{\Z}{\mathbb{Z}}
\newcommand{\coloneqq}{:=}
\newcommand{\id}{\textnormal{id}}
\newcommand{\Exp}{\textnormal{Exp}}

\newcommand{\ctrl}{U}
\newcommand{\dist}{\mathcal{D}}
\newcommand{\st}{X}
\newcommand{\att}{A}
\newcommand{\nbhd}{N}
\newcommand{\open}{W}
\newcommand{\na}{\nbhd\ct\att}%{U_\att}
\newcommand{\vo}{F}
\newcommand{\vt}{G}
\newcommand{\vh}{E}
\newcommand{\vot}{\tilde{F}}
\newcommand{\vtt}{\tilde{G}}
\newcommand{\ub}{u}
\newcommand{\bas}{B}
\newcommand{\lyap}{V}
\newcommand{\lsub}{S}

\newcommand{\Hom}{H}
\newcommand{\Hc}{\check{H}}
\newcommand{\T}{T}

\spnewtheorem*{AC}{Adversary condition}{\bfseries}{\rmfamily}
\spnewtheorem*{HC}{Homology condition}{\bfseries}{\rmfamily}

\spnewtheorem*{AT}{Adversary Theorem}{\bfseries}{\rmfamily}
\spnewtheorem*{HT}{Homology Theorem}{\bfseries}{\rmfamily}

\begin{document}

\title{Relationships Between Necessary Conditions for Feedback Stabilizability}
\titlerunning{Relationships Between Necessary Conditions for Feedback Stabilizability}

\author{Matthew D. Kvalheim}

\institute{Department of Mathematics and Statistics, University of Maryland, Baltimore County, USA
\newline \texttt{kvalheim@umbc.edu}
}

\maketitle

%\tableofcontents

\begin{abstract}
The author's extensions of Brockett's and Coron's necessary conditions for stabilizability are shown to be independent in the fiber bundle picture of control, but the latter is shown to be stronger in the vector bundle picture if the state space is orientable and the \v{C}ech-Euler characteristic of the set to be stabilized is nonzero.
\end{abstract}

\section{Introduction}\label{sec:intro}

Consider a control system 
\begin{equation}\label{eq:cont-sys-intro}
\dot{x} = f(x,u)
\end{equation}
depending on a control parameter $u$, where $x$ lives in a smooth manifold $\st$.
Given a compact subset $\att \subset \st$, the \concept{feedback stabilizability problem} is to determine if there exists a feedback law $x\mapsto u(x)$ of suitable regularity such that $\att$ is an asymptotically stable invariant set for the closed-loop vector field $x\mapsto f(x,u(x))$.

Brockett introduced a technique for proving that such a system is \emph{not} stabilizable for the classical case that $\att$ is a point and $\st = \R^n$.
Viewing $f(x,u)$ as $\R^n$-valued, \concept{Brockett's condition} is satisfied if the equation 
\begin{equation}
v = f(x,u)
\end{equation}
admits a solution $(x,u)$ for each constant vector  $v\in \R^n$ with $\|v\|$ close enough to zero.
Brockett proved that this condition is necessarily satisfied if there exists a continuously differentiable closed-loop vector field solving the feedback stabilizability problem \cite[Thm~1.(iii)]{brockett1983asymptotic}.
In fact, Zabczyk showed that Brockett's condition must hold under the relaxed assumption that the closed-loop vector field is merely continuous and has unique maximal trajectories \cite{zabczyk1989some}.\footnote{In fact, only uniqueness of forward trajectories was assumed in \cite{zabczyk1989some}, but uniqueness of backward trajectories will also be assumed herein, as was done in \cite{kvalheim2022necessary,kvalheim2023obstructions}.}

A monograph by Krasnosel'ski\u{i} and Zabre{i}ko, appearing around the same time as Brockett's paper, contains a necessary condition for asymptotic stability of an equilibrium point of a vector field involving its index \cite[\S 52]{krasnoselskii1984geometrical}.
This condition was used by Coron to introduce a necessary condition for essentially the same setting as Brockett's with $n>1$.
Defining $\Sigma$ to be the complement of $f^{-1}(0)$ within the domain of $f$, \concept{Coron's condition} is satisfied if the $f$-induced homomorphism
\begin{equation}\label{eq:homomorph-intro}
f_*\colon \Hom_{n-1}(\Sigma)\to \Hom_{n-1}(\R^n\ct \{0\})\cong \Z
\end{equation}
of singular homology groups is  surjective. 
Coron proved that this condition must hold if $\att$ is stabilizable in the relaxed sense above  \cite[Thm~2]{coron1990necessary}.
He also proved that Brockett's condition holds whenever Coron's does, and that sometimes Brockett's condition holds while Coron's does not \cite{coron1990necessary}.
Thus, Coron's condition yields a more informative technique for establishing non-stabilizability (but Brockett's condition can be simpler to test in examples).

Subsequently, Mansouri extended Coron's condition to the case that $\att$ is a compact connected smoothly embedded submanifold of $\R^n$, with Euler characteristic denoted by $\chi(\att)\in \Z$.
\concept{Mansouri's condition} is satisfied if the image of \eqref{eq:homomorph-intro} contains the subgroup $\chi(\att)\Z$ if $\dim \att < n-1$, and the subgroup $\chi(M_\att)\Z$ if $\dim A = n-1$ with $M_\att$ the compact domain bounded by $\att$. 
This condition extends Coron's since a point has Euler characteristic equal to $1$.
Mansouri proved that his condition must hold if $\att$ is stabilizable in the same relaxed sense \cite[Thm~4]{mansouri2007local}, \cite[Thm~2.3, Rem.~2.2.1]{mansouri2010topological}.\footnote{Using the local nature of asymptotic stability, Coron \cite{coron1990necessary} and Mansouri \cite{mansouri2007local,mansouri2010topological} actually formulated sharper conditions involving certain subsets of $\Sigma$, and Brockett's condition admits a similar sharpening \cite{zabczyk1989some}. 
Such sharpenings are present in the results of \S \ref{sec:cont-sys-feed-stab}, but do not appear in this introduction for simplicity.}
However, this condition gives no information if $\chi(\att)=0$ or $\chi(M_\att) = 0$, since it then asserts only that the image of \eqref{eq:cont-sys-intro} contains the zero subgroup (as it always does).

Robotics and other modern applications lead to control systems on non-Euclidean state spaces $\st\neq \R^n$ \cite{leve2024future}. 
Moreover, the somewhat-dual problem of determining whether a ``safe'' feedback law exists often leads to stabilizability problems for which $\att$ is not even a manifold.
Motivated by these considerations, the author and Koditschek extended Brockett's condition to one that is necessary for stabilizability of an arbitrary compact subset $\att$ of a smooth manifold $\st$, assuming that $\chi(\att)\neq 0$ \cite[Thm~3.2]{kvalheim2022necessary}.
Here $\chi(\att)$ coincides with the usual Euler characteristic if, e.g., $\att$ is a  manifold or admits a triangulation, but a general definition accounting for all cases is given in \S \ref{sec:prelim}.
Subsequently, the author extended Coron's and Mansouri's conditions to one that is necessary for stabilizability of any compact subset $\att$ of a smooth manifold $\st$, without any Euler characteristic assumptions \cite[Thm~2]{kvalheim2023obstructions}.

The purpose of the present document is to clarify the relationship between these extensions.
As mentioned, Coron showed that his condition is stronger than Brockett's.
This leads one to expect that the extension of Coron's and Mansouri's conditions is stronger than the extension of Brockett's.
Perhaps surprisingly, it is shown (Example~\ref{ex:indep}) that this is not the case for control systems in the ``fiber bundle picture'' introduced by Brockett to account for the situation that the admissible control inputs depend on the state \cite[p.~15]{brockett1977control}, \cite[\S 4.6]{bloch2015nonholonomic}.
In this picture the combined space of states and controls is not globally a product space, and ``$(x,u)$'' in \eqref{eq:cont-sys-intro} is a local representation of a point in a fiber bundle over $\st$.
However, if this fiber bundle is a vector bundle and $\st$ is orientable, it is shown (Theorem~\ref{th:vb-hom-implies-adv}) that the extension of Coron's and Mansouri's conditions is indeed stronger than the extension of Brockett's when $\chi(\att)\neq 0$.
This is the main result herein.

After preliminaries in \S \ref{sec:prelim}, a class of control systems extending the fiber bundle picture is defined in \S\ref{sec:cont-sys-feed-stab}.
The extended necessary conditions are reviewed in \S \ref{sec:nec-cond}, and the relationship between them is clarified in \S \ref{sec:relationships}.

A recent survey of all necessary conditions discussed above is \cite[Ch.~6]{jongeneel2023topological}. 
It should be noted that exponential \cite{gupta2018linear}, global \cite{byrnes2008brockett,baryshnikov2023topological}, time-varying \cite{coron1992global}, and discontinuous \cite{clarke1997asymptotic} variants of the feedback stabilizability problem are related topics of interest, but they are not considered herein. 

\section{Preliminaries}\label{sec:prelim}

The tangent space of a smooth manifold $\st$ at $x\in \st$ is denoted by $\T_x \st$, and the tangent bundle of $\st$ is denoted by $\T \st = \bigcup_{x\in \st}\T_x \st$.
Given a continuous map $\pi\colon \ctrl\to \st$, $\ctrl_x\coloneqq \pi^{-1}(x)$ denotes the fiber over $x\in \st$.
Given a subset $Y\subset \st$, the notations $\T_Y\st \coloneqq \bigcup_{x\in Y}\T_x \st$ and $\ctrl_Y\coloneqq \bigcup_{x\in Y}\ctrl_x$ are in effect.

A continuous vector field $\vt$ on $\st$ will be called \concept{uniquely integrable} if it has unique maximal trajectories.
By the Picard-Lindel\"{o}f theorem, $\vt$ is automatically uniquely integrable if $\vt$ is locally Lipschitz (e.g., smooth).

A subset $\att\subset \st$ is \concept{invariant} for such a $\vt$ if each (forward and backward) $\vt$-trajectory initialized in $\att$ is contained in $\att$.
A compact invariant set $\att\subset \st $ is \concept{asymptotically stable} for such a $\vt$ if for every open set $\open_1 \supset \att$ there is a smaller open set $\open_2\supset \att$ such that every maximal forward $\vt$-trajectory that starts in $\open_2$ stays in $\open_1$ and converges to $\att$.
In this case, a \concept{proper smooth Lyapunov function} always exists on the set of initial conditions converging to $\att$, namely, the \concept{basin of attraction} $\bas$ of $\att$ for $\vt$ \cite[Thm~3.2]{wilson1969smooth}, \cite[\S 6]{fathi2019smoothing}.
By definition, this is a smooth function $\lyap\colon \bas\to [0,\infty)$ with compact sublevel sets such that $\att = \lyap^{-1}(0)$ and $d\lyap \cdot \vt < 0$ on $\bas \ct \att$.

Given an integer $k\geq 0$, $\Hc^k(Z;\R)$ denotes the $k$-th real \concept{\v{C}ech(-Alexander-Spanier) cohomology} of a topological space $Z$ \cite{spanier1966algebraic,dold1972lectures, massey1978homology,bredon1993topology}.
Brief introductions that are more than adequate for the results herein are \cite[\S 3]{gobbino2001topological}, \cite[pp.~371--375]{massey1991basic}.
However, the only facts needed are described below.
First, $\Hc^k(Z;\R)$ is a real vector space.
If it is finite-dimensional for each $k$ and zero for all but finitely many $k$, then the \concept{\v{C}ech-Euler characteristic}
\begin{equation}\label{eq:cech-euler}
\chi(Z)\coloneqq \sum_{k=0}^\infty (-1)^k \dim \Hc^k(Z;\R) 
\end{equation}
is a well-defined integer.

The importance of the \v{C}ech theory for stability theory is as follows.
If a compact invariant set $\att\subset \st$ is asymptotically stable for a uniquely integrable continuous vector field $\vt$ with basin of attraction $\bas$, then the \v{C}ech cohomology of $\att$ is isomorphic to that of $\bas$ \cite[p.~28]{shub1974dynamical}, \cite[Rem.~2.3(b)]{hastings1979higher}, \cite[Thm~6.3]{gobbino2001topological}.
Moreover, if $\lyap\colon \bas\to [0,\infty)$ is a proper smooth Lyapunov function, these \v{C}ech cohomologies are also isomorphic to that of any sublevel set $\lsub\coloneqq \lyap^{-1}([0,c])$ \cite[Prop.~1]{kvalheim2022necessary}, so they are finite-dimensional for each $k$ and vanish for all but finitely many $k$ as above.
Thus, the \v{C}ech-Euler characteristics $\chi(\att)=\chi(\bas) = \chi(\lsub)$ are well-defined and equal.
This is important because the singular homology of $\att$ is not generally isomorphic to that of $\bas$ or $\lsub$, so the usual Euler characteristics need not coincide, as illustrated by the example of the ``Warsaw circle'' \cite[Ex.~3.3]{hastings1979higher} \cite[Ex.~III.8]{jongeneel2023cofibration}.
In particular, $\att$ is not homotopy equivalent to $\bas$ in general, although it is if, e.g., $\att$ is a smoothly embedded submanifold of $\st$ \cite[Thm~III.6]{jongeneel2023cofibration}.
In this case, or more generally if $\att$ is a polyhedron or CW complex, $\chi(\att)$ coincides with the usual Euler characteristic.
In particular, given a triangulation of $\att$, $\chi(\att)$ equals the usual alternating sum $$\#(\text{vertices})-\#(\text{edges})+\#(\text{faces})-\ldots$$ 
An advantage of \v{C}ech cohomology in stability theory is that it enables clean statements valid for \emph{any} compact $\att\subset \st$  without further assumptions.

\section{Control Systems and Feedback Stabilizability}\label{sec:cont-sys-feed-stab}
The framework for the feedback stabilizability problem herein is the following.

A \concept{state space} is a smooth manifold $\st$, a \concept{control space} is a topological space $\ctrl$ equipped with a continuous map $\pi\colon \ctrl \to \st$, and a \concept{control system} is a continuous map $f\colon \ctrl\to \T \st$ that sends each fiber $\ctrl_x= \pi^{-1}(x)$ into $\T_x \st$.

A section of $\pi$ is a map $\ub\colon \st\to \ctrl$ satisfying $\pi\circ \ub = \id_\st$.
A \concept{feedback law} is a continuous section $\ub$ of $\pi$ such that the \concept{closed-loop vector field} $f\circ \ub$ is uniquely integrable.
A compact set $\att\subset \st$ is \concept{stabilizable} if there is a feedback law $\ub$ such that $\att$ is an asymptotically stable invariant set for $f\circ \ub$.

\begin{remark}\label{rem:cont-sys-framework}
The definition of control systems here reduces to the fiber bundle picture of control in the special case that $\pi\colon \ctrl \to \st$ is a fiber bundle and $f\colon \ctrl \to \T \st$ is sufficiently smooth \cite[p.~15]{brockett1977control}, \cite[\S 4.6]{bloch2015nonholonomic}.
If $\pi$ is a trivial (product) bundle $\pi \colon \ctrl = \st \times Y\to \st$, then elements of $\ctrl$ are of the form $(x,u)$, so the expression ``$f(x,u)$'' makes sense.
Since elements of $\ctrl$ are not of this form in the general framework here, it seems appropriate to instead write ``$f(\ub)$'', keeping in mind that $\ub\in \ctrl$ contains all information about the state $x=\pi(\ub)$.
This is why the notation $f\circ \ub$ is employed here for feedback laws $\ub\colon \st\to \ctrl$.
\end{remark}

\section{Necessary Conditions for Feedback Stabilizability}\label{sec:nec-cond} 
Fix a control system $f\colon \ctrl \to \T \st$ and compact set $\att \subset \st$ that is an asymptotically stable invariant set for \emph{some} uniquely integrable continuous vector field $\vt$ on $\st$, so that $\chi(\att)$ is well-defined (\S\ref{sec:prelim}).
Note that $\vt$ has nothing to do with $f$.

An ``adversary condition'' and a ``homology condition'' are introduced below.
These conditions do not always hold, but Theorems~\ref{th:gen-brockett}, \ref{th:homology} show that they must hold if $\att$ is stabilizable (assuming also that $\chi(\att)\neq 0$ for Theorem~\ref{th:gen-brockett}).

The first condition extends Brockett's.

\begin{AC}
For any neighborhood $\nbhd_\att \subset \st$ of $\att$ there is a neighborhood $\nbhd_0 \subset \T \st$ of the zero section such that, for any $\nbhd_0$-valued continuous vector field $\vo$ over $\nbhd_\att$, $\vo(x)=f(\ub_x)$ for some $x\in \nbhd_\att$,  $\ub_x\in \ctrl_x$.
\end{AC}

The terminology is explained by interpreting this as the ability to ``\concept{defeat}'' \concept{adversaries} $\vo$ via the image of $f$ somewhere intersecting the image of $\vo$ \cite{kvalheim2022necessary}.

\begin{remark}\label{rem:brock-specialization}
The adversary condition implies Brockett's on taking $\st = \R^n$, $\pi\colon \ctrl = \R^n \times \R^m\to \R^n$ to be a trivial bundle, $\att$ to be a point, and adversaries $\vo$ to be constant.
If instead adversaries are not required to be constant, one obtains a condition that can fail even when Brockett's does not \cite[\S 6.1]{kvalheim2022necessary}.
\end{remark}

The following theorem is due to the author and Koditschek \cite[Thm~3.2]{kvalheim2022necessary}.

\begin{theorem}\label{th:gen-brockett}
If $\chi(\att)\neq 0$ and $\att$ is stabilizable, the adversary condition holds.
\end{theorem}

See Figure~\ref{fig:adversary-theorem} for an illustration of the proof.

\begin{proof}
Fix a neighborhood $\nbhd_\att\subset \st$ of $\att$.
Assume that $\chi(\att)\neq 0$ and $\att$ is stabilizable.
Let $\ub\colon \st\to \ctrl$ be a stabilizing feedback law, so that $\att$ is asymptotically stable for the uniquely integrable continuous vector field $f\circ \ub$.
Let $\lyap$ be a proper smooth Lyapunov function for $\att$ with respect to $f\circ \ub$ and $c>0$ be small enough that the compact smooth manifold with boundary $\lsub\coloneqq \lyap^{-1}([0,c])\subset \nbhd_\att$.
Since  $f\circ \ub$ points into the interior of $\lsub$ at its boundary, so does $f\circ\ub - \vo$ for any continuous vector field $\vo$ over $\nbhd_\att$ taking values in a sufficiently small neighborhood $\nbhd_0\subset \T \st$ of the zero section. 
Since also $\chi(\lsub) = \chi(\att)\neq 0$ (\S \ref{sec:prelim}), the Poincar\'{e}-Hopf  theorem implies that $f\circ \ub(x) - \vo(x) = 0$ for some $x\in \st$ \cite{milnor1965topology,pugh1968generalized,guillemin1974differential}.
Defining $\ub_x\coloneqq \ub(x)$ finishes the proof.
\end{proof}

\begin{figure}
	\centering
   	\def\svgwidth{0.65\columnwidth}
%% Accompanies image file 'gen-brockett-proof.pdf' (pdf, eps, ps)
\begingroup%
  \makeatletter%
  \providecommand\color[2][]{%
    \errmessage{(Inkscape) Color is used for the text in Inkscape, but the package 'color.sty' is not loaded}%
    \renewcommand\color[2][]{}%
  }%
  \providecommand\transparent[1]{%
    \errmessage{(Inkscape) Transparency is used (non-zero) for the text in Inkscape, but the package 'transparent.sty' is not loaded}%
    \renewcommand\transparent[1]{}%
  }%
  \providecommand\rotatebox[2]{#2}%
  \newcommand*\fsize{\dimexpr\f@size pt\relax}%
  \newcommand*\lineheight[1]{\fontsize{\fsize}{#1\fsize}\selectfont}%
  \ifx\svgwidth\undefined%
    \setlength{\unitlength}{218.02372356bp}%
    \ifx\svgscale\undefined%
      \relax%
    \else%
      \setlength{\unitlength}{\unitlength * \real{\svgscale}}%
    \fi%
  \else%
    \setlength{\unitlength}{\svgwidth}%
  \fi%
  \global\let\svgwidth\undefined%
  \global\let\svgscale\undefined%
  \makeatother%
  \begin{picture}(1,0.79628066)%
    \lineheight{1}%
    \setlength\tabcolsep{0pt}%
    \put(0,0){\includegraphics[width=\unitlength,page=1]{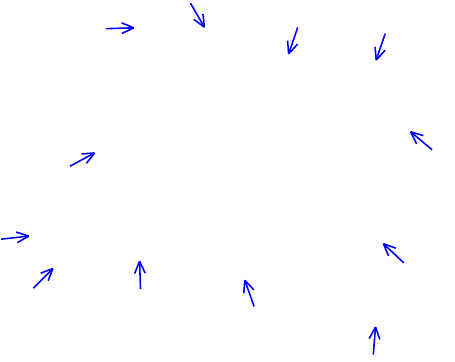}}%
    \put(0.45310083,0.27818989){\makebox(0,0)[lt]{\lineheight{1.25}\smash{\begin{tabular}[t]{l}$\att$\end{tabular}}}}%
    \put(0.46686111,0.59466921){\makebox(0,0)[lt]{\lineheight{1.25}\smash{\begin{tabular}[t]{l}$\lsub$\end{tabular}}}}%
    \put(0.06782183,0.29883026){\color[rgb]{0,0,1}\makebox(0,0)[lt]{\lineheight{1.25}\smash{\begin{tabular}[t]{l}$f\circ \ub$\end{tabular}}}}%
    \put(0,0){\includegraphics[width=\unitlength,page=2]{gen-brockett-proof.pdf}}%
    \put(0.06164889,0.21936868){\color[rgb]{0,0.50196078,0}\makebox(0,0)[lt]{\lineheight{1.25}\smash{\begin{tabular}[t]{l}$f\circ u - \vo$\end{tabular}}}}%
    \put(0,0){\includegraphics[width=\unitlength,page=3]{gen-brockett-proof.pdf}}%
  \end{picture}%
\endgroup%
	
	\caption{An illustration of the proof of Theorem~\ref{th:gen-brockett}.}\label{fig:adversary-theorem}
\end{figure}

\begin{example}\label{ex:gen-brockett}
Assume there are $\dist \subset \T\st$ and a continuous vector field $\vo$ such that, for each $x\in \st$: $\dist_x\subset \T_x \st$ is a vector subspace of positive codimension, $f(\ctrl_x)\subset \dist_x$, and $\vo(x)\not \in \dist_x$.\footnote{If $\dist \subset \T \st$ is a vector subbundle of positive corank, then a continuous vector field $\vo$ satisfying $\vo(x)\not\in \dist_x$ for all $x\in \st$ exists if and only if the quotient bundle $\T\st /\ctrl$ admits a nowhere-zero section \cite[Thm~4.2.2]{hirsch1976differential}, which is always the case if $\st$ is contractible (e.g., $\st = \R^n$), since then  $\T \st/\ctrl$ is trivializable \cite[Cor.~4.2.5]{hirsch1976differential}.\label{foot:normal-vector-bundle-section}}
Then $\vo_\varepsilon(x)\coloneqq \varepsilon\vo(x)\not\in \dist_x$ for each $x\in \st$ and $\varepsilon \neq 0$, so the adversary condition does not hold.
Hence $\chi(\att)=0$ if $\att$ is stabilizable (Theorem~\ref{th:gen-brockett}).
In particular, if $\st=\R^3$ or $\st = \sph^1\times \R^2$, then any stabilizable compact connected embedded submanifold $\att$ is homeomorphic to a circle or a torus\footnote{This is because $\sph^1\times \R^2$ embeds in $\R^3$, and the only compact connected embedded submanifolds $\att$ of $\R^3$ with $\chi(\att)=0$ are topological circles and tori \cite[pp.~9, 32--33]{massey1991basic}.}---in contrast, Brockett's theorem only implies that $\att$ is not a point \cite[Thm~1.(iii)]{brockett1983asymptotic}.
Special cases of the present example include the \emph{Heisenberg system} (or \emph{nonholonomic integrator}) \cite{brockett1983asymptotic,bloch2015nonholonomic} with $\st = \R^3$ and the standard \emph{kinematic unicycle} model from robotics \cite{choset2005principles,vasilopoulos2022reactive} with $\st = \sph^1\times \R^2$.
\end{example}

The next condition extends Coron's and Mansouri's.
Given a subset $Y\subset \st$, the notation $\Sigma_{Y}\coloneqq \ctrl_{Y} \ct f^{-1}(0)$ will be used from here on.
\begin{HC}
For all small enough neighborhoods $\nbhd\subset \st$ of $\att$, 
\begin{equation}\label{eq:HC}
\vt_*\Hom_{\bullet}(\na) \subset f_*\Hom_{\bullet}(\Sigma_{\na}) \textnormal{ in } \Hom_{\bullet}(\T_{\na}\st \ct 0_{\T \st}),
\end{equation}
where $\vt_*$, $f_*$ are the induced graded homomorphisms on singular homology.
\end{HC}

\begin{remark}\label{rem:c-m-specialization}
The homology condition implies Mansouri's upon taking $\st = \R^n$, $\pi\colon \ctrl = \R^n \times \R^m\to \R^n$ to be a trivial bundle, $\att$ to be a compact connected smoothly embedded submanifold, and adversaries $\vo$ to be constant.
Indeed, let $p_2\colon \T \R^n \approx \R^n \times \R^n \to \R^n$ be the projection onto the second factor and note that, if $\nbhd$ is small enough that $\vt|_{\na}$ is nowhere zero, \eqref{eq:HC} implies 
\begin{equation}\label{eq:rem-c-m-specialization}
(p_2\circ \vt)_*\Hom_{n-1}(\na) \subset (p_2\circ f)_*\Hom_{n-1}(\Sigma_{\na}) \textnormal{ in } \Hom_{n-1}(\R^n \ct \{0\}).
\end{equation}
Since $\R^n\ct \{0\}$ deformation retracts onto the unit sphere $\sph^{n-1}\subset \R^n$, degree theory \cite[p.~28]{milnor1965topology} and the Poincar\'{e}-Hopf theorem applied to a sublevel set $\lsub\subset \nbhd$ of a proper smooth Lyapunov function for $\vt$ imply that the left side of \eqref{eq:rem-c-m-specialization} contains $\chi(\att)\Z$ if $\dim \att < n-1$.
If instead $\dim \att = n-1$, then the unique compact domain $M_\att$ bounded by $\att$ \cite[p.~89]{guillemin1974differential} is also asymptotically stable for $\vt$, so a similar argument shows that the left side of \eqref{eq:rem-c-m-specialization} contains $\chi(M_\att)\Z$.
Thus, Mansouri's condition holds if the homology condition holds.
When $\att\subset \R^n$ is a point and $n>2$, Mansouri's condition specializes to Coron's except that the latter contains a refinement made possible by the fact that the closed-loop vector field vanishes at the equilibrium $\att$.
The technique used to prove Theorem~\ref{th:homology} below can also be used to prove Coron's theorem that this refined condition is necessary for stabilizability.
\end{remark}

The proof that the homology condition is necessary for stabilizability relies on the following \concept{homotopy theorem} \cite[Thm~1]{kvalheim2023obstructions}. % not proved herein.
Figure~\ref{fig:homotopy-theorem} illustrates the basic idea of the proof, which is not given here.
See the proof in \cite{kvalheim2023obstructions} for details. 

\begin{theorem}\label{th:homotopy}
If $\att$ is an asymptotically stable invariant set for a uniquely integrable continuous vector field $\vo$ in addition to $\vt$, there is an open set $\open\supset \att$ such that $\vo|_{\open \ct \att}$ and $\vt|_{\open \ct \att}$ are homotopic through nowhere-zero vector fields.
\end{theorem}
\begin{figure}
	\centering
   	\def\svgwidth{1.0\columnwidth}
   	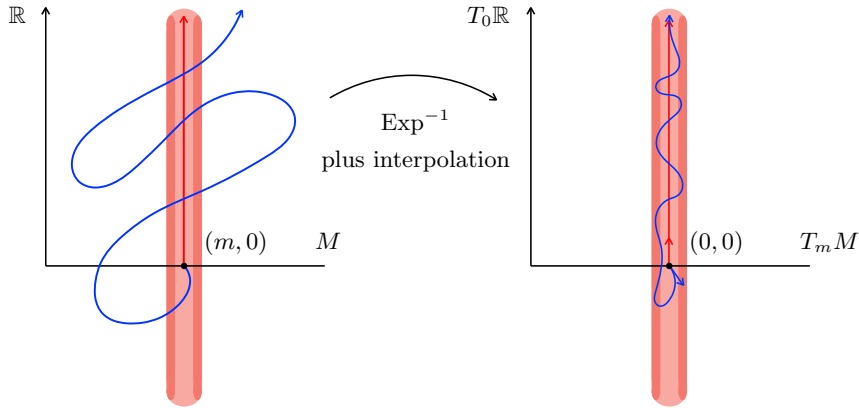		
	\caption{An illustration of the idea of proof of the homotopy theorem (Theorem~\ref{th:homotopy}).
	Using a level set $M$ of a proper smooth Lyapunov function, the complement $\bas\setminus \att$ of $\att$ in its basin of attraction $\bas$ for $\vt$ is identified with $M\times \R$ via a diffeomorphism identifying $\vt$-trajectories with straight lines $t\mapsto (m,t)$ converging to $\att$ as $t\to\infty$ \cite[Thm~3.2]{wilson1967structure}.
	The product of any Riemannian metric on $M$ with the Euclidean metric on $\R$ induces an exponential map $\Exp\colon \T(\bas \setminus \att)\to (\bas\setminus \att)^2$ that restricts to a diffeomorphism from a neighborhood (comprising a family of ``tubes'', as on the right of the figure) of $\textnormal{span}(\vt)\approx \T \R\times 0_{\T M}$ onto its image.
	The inverse $\Exp^{-1}$ is used to convert the motion of $\vo$- and $\vt$-trajectories starting near $\att$ (left side of the figure) to homotopies of $\vo$ and $\vt$ through nowhere-zero vector fields to a common vector field.
	To account for the problem of $\vo$-trajectories leaving the domain of $\Exp^{-1}$, and to ``finish off'' the homotopies to a common vector field, the $\Exp^{-1}$-converted $\vo$-trajectories are smoothly interpolated toward those of $\vt$ (i) as $\vo$-trajectories leave the domain and (ii) after a sufficiently large time. 	
	(In the proof given in \cite[Thm~1]{kvalheim2023obstructions}, a homotopy from $\vo$ to $\vt$ is constructed directly; the alternative proof sketched here using homotopies to a common third vector field was chosen for ease of illustration.)
	}\label{fig:homotopy-theorem}
\end{figure}

The following theorem is a slight extension (to non-open $\nbhd$) of \cite[Thm~2]{kvalheim2023obstructions}.

\begin{theorem}\label{th:homology}
If $\att$ is stabilizable, then the homology condition holds.
\end{theorem}
\begin{remark}
Theorem~\ref{th:homology} remains true (by the same proof) for homology with coefficients in an abelian group, or with homotopy groups instead of homology.
\end{remark}

\begin{proof}
Let $\ub\colon \st\to \ctrl$ be a stabilizing feedback law, so that $\att$ is asymptotically stable for the uniquely integrable continuous vector field $f\circ \ub$.
The homotopy theorem applied to $\vo = f\circ \ub$ implies the existence of an open set $\open\supset \att$ such that, for any neighborhood $\nbhd \subset \open$ of $\att$,  $\vo|_{\na}$ and $\vt|_{\na}$ are homotopic through nowhere-zero vector fields over $\na$ (restrict the homotopy of Theorem~\ref{th:homotopy}).
Thus, 
$\vo_*\Hom_{\bullet}(\na) = \vt_*\Hom_{\bullet}(\na) \textnormal{ in } \Hom_{\bullet}(\T_{\na}\st \ct 0_{\T \st}).$
Since $\vo_*\Hom_{\bullet}(\na) = f_*\circ \ub_* \Hom_{\bullet}(\na)$, this yields \eqref{eq:HC} and finishes the proof.
\end{proof}

It is interesting to compare the assumptions of the following example with those of Example~\ref{ex:gen-brockett}.

\begin{example}\label{ex:2}
Assume $\st\ct \att$ is orientable and there are $\dist \subset \T\st$ and a continuous vector field $\vo$ such that, for each $x\in \st\ct \att$: $\dist_x\subset \T_x \st$ is a vector subspace of positive codimension, $f(\ctrl_x)\subset \dist_x$, and $\vo(x)\not \in \dist_x$.
Since $\att$ is asymptotically stable for $\vt$, there is an arbitrarily small compact neighborhood $\nbhd$ of $\att$ such that $\vt|_{\na}$ is nowhere-zero, $\partial \nbhd$ is a smoothly embedded hypersurface, and $\na$ is diffeomorphic to $(0,1]\times \partial \nbhd$ \cite[Thm~3.2]{wilson1967structure}. With $n\coloneqq\dim \st$ and $\{1\} \times \partial \nbhd$ identified with $\partial \nbhd\subset \na$, assume 
\begin{equation}\label{eq:ex-2-1}
\vt_*([M]) \not = \vo_*([M]) \textnormal{ in } \Hom_{n-1}(\T_{M}\st \ct 0_{\T \st})
\end{equation}
for some connected component $M$ of $\partial \nbhd$, where $[M]\in \Hom_{n-1}(M)$ is the fundamental class \cite[p.~355]{bredon1993topology} of $M$.

Then $\att$ is not stabilizable (Theorem~\ref{th:homology}).
Indeed, \eqref{eq:ex-2-1} implies that
\begin{equation*}%\label{eq:ex-2-1-5}
\vt_*\Hom_{n-1}(\partial \nbhd) \not \subset \vo_*\Hom_{n-1}(\partial \nbhd) \textnormal{ in } \Hom_{n-1}(\T_{\partial \nbhd}\st \ct 0_{\T \st})
\end{equation*}
since $\Hom_{n-1}(M)\cong \Z[M]$, $\vo$ and $\vt$ are vector fields, and the homology of a space is the direct sum of the homology of its connected components (see the proof of Lemma~\ref{lem:hv} in \S\ref{sec:relationships} for further related details).
Since $\nbhd\setminus \att$ deformation retracts onto $\partial \nbhd$, homotopy invariance of homology \cite[Cor.~IV.16.5]{bredon1993topology} yields
$$\vt_*\Hom_{n-1}(\na) \not \subset \vo_*\Hom_{n-1}(\na) \textnormal{ in } \Hom_{n-1}(\T_{\na}\st \ct 0_{\T \st}),$$
which in turn implies that
\begin{equation}\label{eq:ex-2-a}
\vt_*\Hom_{n-1}(\na) \not \subset (\vo\circ \pi)_*\Hom_{n-1}(\Sigma_{\na}) \textnormal{ in } \Hom_{n-1}(\T_{\na}\st \ct 0_{\T \st})
\end{equation}
since $(\vo\circ \pi)_*\Hom_{n-1}(\Sigma_{\na})=\vo_*\pi_*\Hom_{n-1}(\Sigma_{\na})\subset \vo_*\Hom_{n-1}(\na)$.
And since $f|_{\Sigma_{\na}}$, $\vo\circ \pi|_{\Sigma_{\na}}$ are straight-line homotopic as maps into $\T_{\na}\st \ct 0_{\T \st}$, homotopy invariance and \eqref{eq:ex-2-a} imply that
\begin{equation*}%\label{eq:ex-2-b}
\vt_*\Hom_{n-1}(\na) \not \subset f_*\Hom_{n-1}(\Sigma_{\na}) \textnormal{ in } \Hom_{n-1}(\T_{\na}\st \ct 0_{\T \st}).
\end{equation*}
Thus, the homology condition does not hold since $\nbhd$ was arbitrarily small, so Theorem~\ref{th:homology} establishes the claim that $\att$ is not stabilizable. 
\end{example} 

\begin{example}
Let all assumptions be as in Example~\ref{ex:2}, but further assume that $\st = \R^n$ and $\partial \nbhd$ is connected (such an $\nbhd$ always exists if, e.g.,  $\att$ is a smoothly embedded submanifold of codimension $\geq 2$ \cite[Thm~3.2, Cor.~3.5]{wilson1967structure}), and instead of \eqref{eq:ex-2-1} assume that
\begin{equation}\label{eq:ex-3-1}
 \deg\left(\frac{p_2\circ\vo}{\|p_2\circ\vo\|}|_{\partial \nbhd}\right)\neq \chi(\att).
\end{equation}
Here $p_2\colon \T \R^n \approx \R^n \times \R^n \to \R^n$ is the projection onto the second factor, and $\deg(\vh)$ is the Brouwer degree\footnote{After replacing $\vh$ by a homotopic smooth map \cite[p.~70]{guillemin1974differential}, $\deg(\vh)$ can be defined concretely as a signed count of points in the preimage $\vh^{-1}(y)$ for a ``typical'' $y\in \sph^{n-1}$ \cite[p.~27]{milnor1965topology}, \cite[pp.~108--109]{guillemin1974differential}.} of a continuous map $\vh\colon \partial \nbhd\to \sph^{n-1}$.

Then \eqref{eq:ex-2-1} holds (with $M=\partial \nbhd$), so $\att$ is not stabilizable (by Example~\ref{ex:2}).

Indeed, $\att$ is asymptotically stable for a smooth vector field $\tilde{\vt}$ that points inward on $\partial \nbhd$ \cite[p.~327]{wilson1967structure}, so (after shrinking $\nbhd$ if necessary) the homotopy theorem implies $\tilde{\vt}|_{\na}, \vt|_{\na}$ are homotopic through nowhere-zero vector fields.
This, degree theory \cite[p.~28]{milnor1965topology}, and the Poincar\'{e}-Hopf theorem imply
\begin{equation}\label{eq:ex-3-2}
 \deg\left(\frac{p_2\circ\vt}{\|p_2\circ\vt\|}|_{\partial \nbhd}\right)=  \deg\left(\frac{p_2\circ\tilde{\vt}}{\|p_2\circ\tilde{\vt}\|}|_{\partial \nbhd}\right) = \chi(\att).
\end{equation}
Since $\vh_*([\partial \nbhd])$ is $\deg(\vh/\|\vh\|)$ times the canonical generator of $\Hom_{n-1}(\R^n\setminus \{0\})\cong \Z$ for any continuous map $\vh\colon \partial \nbhd\to \R^{n}\setminus \{0\}$, \eqref{eq:ex-3-1} and \eqref{eq:ex-3-2} yield $$(p_2)_*\circ \vt_*([\partial \nbhd])\neq (p_2)_*\circ \vo_*([\partial \nbhd]) \textnormal{ in } \Hom_{n-1}(\R^n\setminus \{0\}),$$
which in turn establishes \eqref{eq:ex-2-1} and hence non-stabilizability of $\att$, as claimed.
\end{example}

\section{Relationships between the conditions}\label{sec:relationships}

The main result is motivated by the following example, which shows that the homology condition does not imply the adversary condition in general.

\begin{example}\label{ex:indep}
Let $\st = \R^2$, $\pi \colon \ctrl \to \st$ be the circle bundle of unit-norm tangent vectors $\ctrl \subset \T \st$, $f\colon \ctrl \hookrightarrow \T \st$ be the inclusion, and $\att = \{0\}$ be the origin.
Any continuous vector field $\vt$ that asymptotically stabilizes $\att$ does not vanish on $\na$ for sufficiently small neighborhoods $\nbhd\subset \st$ of $\att$, so is straight-line homotopic over $\na$ through nowhere-zero vector fields to $\vt/\|\vt\|$.
Thus,
 $$\vt_*\Hom_{\bullet}(\na) = (\vt/\|\vt\|)_*\Hom_{\bullet}(\na) \subset f_*\Hom_{\bullet}(\Sigma_{\na}) \textnormal{ in } \Hom_\bullet(\T_{\na}\st \ct 0_{\T \st}),$$ 
where the set inclusion holds since $G/\|G\| = f\circ \ub$ for $\ub = G/\|G\|$.
Hence the homology condition holds, but the adversary condition does not since the images of $f$ and any adversary $\vo$ with $\|\vo\|<1$ are disjoint.
\end{example}

The homology condition does not imply the adversary condition in Example~\ref{ex:indep} even though $\st$ is orientable, $\chi(\att)\neq 0$, and $\pi\colon\ctrl \to \st$ is a fiber bundle.
But Theorem~\ref{th:vb-hom-implies-adv} below shows that this implication does hold if additionally $\pi\colon \ctrl \to \st$ is a vector bundle.

\begin{lemma}\label{lem:hv}
Let $\lsub$ be a compact oriented smooth $n$-dimensional manifold with nonempty boundary $\partial \lsub$ satisfying $\chi(\lsub)\neq 0$, and $\vo$, $\vt$ be continuous vector fields on $\lsub$ such that $\vo$ is nowhere-zero and $\vt$ points inward at $\partial \lsub$. 
Then 
\begin{equation}\label{eq:lem-hv}
\vt_* \Hom_{n-1}(\partial \lsub)\not \subset \vo_* \Hom_{n-1}(\partial \lsub) \textnormal{ in } \Hom_{n-1}(\T_{\partial \lsub}\lsub\ct 0_{\T \lsub}).
\end{equation}
\end{lemma}

\begin{proof}
Fix a Riemannian metric on $\lsub$ with induced norm $\|\cdot\|$.
After homotopies of $\vo$, $\vt$ through vector fields not vanishing on $\partial \lsub$ it may be assumed that $\vo$, $\vt$ are smooth and $\vt^{-1}(0)$ is finite and disjoint from $\partial \lsub$ \cite[p.~135]{hirsch1976differential}.
Define $\vot\coloneqq \vo/\|\vo\|$ and $\vtt\coloneqq \vt/\|\vt\|$.

By the vector field straightening theorem \cite[p.~62]{bloch2015nonholonomic} there are disjoint open neighborhoods $\open_x$ of $x\in \vt^{-1}(0)$ and charts $\open_x \to \R^n$ pushing forward $\vot$ to constant unit vector fields on $\R^n$.
Let $\{B_x\}_{x\in \vt^{-1}(0)}$ consist of disjoint closed Euclidean balls $B_x\subset \open_x$ centered at $x$,  and define $B\coloneqq \bigcup_{x\in \vt^{-1}(0)}B_x$.

Constancy of $\vot$ in each chart reveals that the oriented intersection number $I(-\vot|_{\partial B},\vtt|_{\partial B})$ \cite{guillemin1974differential,hirsch1976differential} is equal to the sum of the indices of zeros of $\vt$, so
\begin{equation*}%\label{eq:oin-euler}
I(-\vot|_{\partial B},\vtt|_{\partial B}) = \chi(\lsub)\neq 0
\end{equation*}
by the Poincar\'{e}-Hopf theorem \cite{milnor1965topology, guillemin1974differential,hirsch1976differential}.
Since also $\partial B \cup \partial \lsub = \partial (\lsub\ct \textnormal{int}(B))$, 
\begin{equation}\label{eq:oin-boundary}
I(-\vot|_{\partial \lsub},\vtt|_{\partial \lsub}) =  I(-\vot|_{\partial B},\vtt|_{\partial B}) \neq 0.
\end{equation} 

Now suppose that  \eqref{eq:lem-hv} does not hold.
Then 
\begin{equation*}
\vtt_*\Hom_{n-1}(\partial \lsub)\subset \vot_*\Hom_{n-1}(\partial \lsub) \textnormal{ in } \Hom_{n-1}(\T^1_{\partial \lsub}\lsub)\cong \bigoplus_{M\in \pi_0(\partial \lsub)} \Hom_{n-1}(\T^1_M\lsub),
\end{equation*}
where $\T^1_Z\lsub \subset \T_Z \lsub$ denotes the unit sphere bundle over a subset $Z\subset \lsub$ and $\pi_0(Z)$ denotes the set of connected components of $Z$.
Since $$\Hom_{n-1}(\partial \lsub)\cong \bigoplus_{M\in \pi_0(\partial \lsub)} \Hom_{n-1}(M)$$ and $\vot_*, \vtt_*\colon \Hom_{n-1}(\partial \lsub)\to \Hom_{n-1}(\T^1_{\partial \lsub}\lsub)$ are the direct sums of their restrictions $\Hom_{n-1}(M)\to \Hom_{n-1}(\T^1_{M}\lsub)$ \cite[p.~184]{bredon1993topology}, then for each $M\in \pi_0(\partial \lsub)$:
\begin{equation}\label{eq:lem-hv-2}
\vtt_*\Hom_{n-1}(M)\subset \vot_*\Hom_{n-1}(M) \textnormal{ in } \Hom_{n-1}(\T^1_{M}\lsub).
\end{equation}
Let $[M]\in \Hom_{n-1}(M)$ be the fundamental class of $M\in \pi_0(\partial \lsub)$ with orientation induced by that of $\lsub$ \cite[p.~355]{bredon1993topology}.
Since $\Hom_{n-1}(M)\cong \Z[M]$ and the homology homomorphism induced by the bundle projection $\T^1_M\lsub \to M$ sends both $\vot_*([M])$ and $\vtt_*([M])$ to $[M]$, \eqref{eq:lem-hv-2} implies that, for each $M\in \pi_0(\partial \lsub)$:
\begin{equation*}%\label{eq:lem-hv-3}
\vtt_*([M])=\vot_*([M]) \textnormal{ in } \Hom_{n-1}(\T^1_{M}\lsub),
\end{equation*}
which in turn implies $I(-\vot|_{M},\vtt|_{M})=I(-\vot|_{M},\vot|_{M})=0$ \cite[\S VI.11]{bredon1993topology}.
Hence
\begin{align*}
0&= \sum_{M\in \pi_0(\partial \lsub)} I(-\vot|_{M},\vot|_{M})
= \sum_{M\in \pi_0(\partial \lsub)} I(-\vot|_{M},\vtt|_{M})\\
&= I(-\vot|_{\partial \lsub},\vtt|_{\partial \lsub}),
\end{align*}
but this contradicts \eqref{eq:oin-boundary} and completes the proof.
\end{proof}

As in \S \ref{sec:nec-cond}, the main result below concerns a control system $f\colon \ctrl \to \T \st$ (\S \ref{sec:cont-sys-feed-stab}) with smooth $n$-dimensional state space $\st$ and a compact set $\att \subset \st$ that is an asymptotically stable invariant set for \emph{some} uniquely integrable continuous vector field $\vt$ on $\st$, so that $\chi(\att)$ is well-defined (\S\ref{sec:prelim}).

\begin{theorem}\label{th:vb-hom-implies-adv}
Assume that $\pi\colon \ctrl \to \st$ is a vector bundle, $\st$ is orientable, and $\chi(\att) \neq 0$.
If the homology condition holds, then the adversary condition holds.
\end{theorem}

\begin{proof}
To prove the theorem it will be shown that the homology condition does not hold if the adversary condition does not hold.
Fix an orientation and a Riemannian metric on $\st$ with induced norm $\|\cdot\|$.

Assume that the adversary condition does not hold.
Then there is an arbitrarily small neighborhood $\nbhd\subset \st$ of $\att$ such that, for any $\delta_0 > 0$, there is a continuous vector field $\vh_0$ over $\nbhd$ satisfying  $\|\vh_0\| \leq \delta_0$ and $\vh_0(x)\neq f(\ub_x)$ for all $x\in \nbhd$, $\ub_x\in \ctrl_x$.

Let $\lyap$ be a proper smooth Lyapunov function for $\att$ with respect to $\vt$ and $c> 0$ be small enough that $\lsub= \lyap^{-1}([0,c])\subset \nbhd$.
Since $\nbhd$ was arbitrarily small, the homology condition does not hold if
\begin{equation}\label{eq:no-hc-red-1}
\vt_* \Hom_{n-1}(\lsub\ct \att) \not \subset f_* \Hom_{n-1}(\Sigma_{\lsub \ct \att}) \textnormal{ in } \Hom_{n-1}(\T_{\lsub\ct \att}\st\ct 0_{\T \st}).
\end{equation}

Since 
\begin{equation}\label{eq:def-retr}
\lsub \ct \att  \textnormal{ deformation retracts onto } \partial \lsub
\end{equation}
via the flow of $\vt$ \cite[Thm~3.2]{wilson1967structure}, the homotopy lifting property \cite[Cor.~VII.6.12]{bredon1993topology} implies that
\begin{equation}\label{eq:def-retr-tangent}
\T_{\lsub\ct \att}\st\ct 0_{\T \st}  \textnormal{ deformation retracts onto } \T_{\partial \lsub}\lsub \ct 0_{\T\lsub},
\end{equation}
which in turn deformation retracts onto a compact submanifold, namely, the unit sphere bundle in $\T_{\partial \lsub}\lsub$.
Hence $\Hom_{n-1}(\T_{\lsub\ct \att}\st\ct 0_{\T \st})$ is finitely generated \cite[Cor.~IV.16.5, Cor.~VI.7.12]{bredon1993topology}, so there is a compact set $K\subset \Sigma_{\lsub\ct \att}$ such that
\begin{equation}\label{eq:f-Sig-K-hom}
f_* \Hom_{n-1}(\Sigma_{\lsub\ct \att}) = f_* \Hom_{n-1}(K) \textnormal{ in } \Hom_{n-1}(\T_{\lsub\ct \att}\st\ct 0_{\T \st}).
\end{equation}

Defining
$$\delta \coloneqq \frac{1}{2} \min\{\|f(\ub)\|\colon \ub \in K\} > 0,$$
the second paragraph of the proof implies that there is a continuous vector field $\vh$ on $\lsub$ satisfying  $\|\vh\| \leq \delta$ and $\vh(x)\neq f(\ub_x)$ for all $x\in \lsub$, $\ub_x\in \ctrl_x$. 
Let $\varphi\colon \T \lsub \to [0,1]$ be a  continuous function equal to $0$ on $f(K)$ and $1$ on $\{v\in \T\lsub \colon \|v\|\leq \delta\}$.
Then the fiber-preserving map $\tilde{f}\colon \ctrl_{\lsub}\to \T \lsub $ defined by $\tilde{f}\coloneqq f-(\varphi\circ f)\vh\circ \pi$ satisfies $\tilde{f}|_K = f|_K$ and 
\begin{equation}\label{eq:tf-nonzero}
\tilde{f}^{-1}(0_{\T\lsub}) = \varnothing,
\end{equation}
since $\|(\varphi\circ f)E\circ \pi\|\leq \|E\circ \pi\| \leq  \delta$ implies that $\tilde{f}(\ub_x)=0$ if and only if $f(\ub_x)=\vh(x)$, which never occurs by the definition of $\vh$.

Equation \eqref{eq:f-Sig-K-hom} and the fact that $\tilde{f}|_K = f|_K$ imply that $f_*\Hom_{n-1}(\Sigma_{\lsub\ct \att}) = \tilde{f}_*\Hom_{n-1}(K)$.
Since $K\subset \ctrl_{\lsub\ct \att}$, this and \eqref{eq:tf-nonzero} imply that
\begin{equation}\label{eq:f-tf-hom-inclusion}
f_* \Hom_{n-1}(\Sigma_{\lsub \ct \att})\subset \tilde{f}_*\Hom_{n-1}(\ctrl_{\lsub\ct \att}) \textnormal{ in } \Hom_{n-1}(\T_{\lsub\ct \att}\st\ct 0_{\T \st}).
\end{equation}
If $\ub_0\colon \lsub\to \ctrl_\lsub$ denotes the zero section, then \eqref{eq:def-retr} implies that $\ctrl_{\lsub\ct \att}$ deformation retracts onto $\ctrl_{\partial \lsub}$ \cite[Cor.~VII.6.12]{bredon1993topology}, which in turn deformation retracts onto the image  of $\ub_0|_{\partial \lsub}$.
Defining $\vo\coloneqq \tilde{f}\circ \ub_0$, the preceding sentence and \eqref{eq:f-tf-hom-inclusion} imply \cite[Cor.~IV.16.5]{bredon1993topology} that 
\begin{equation}\label{eq:f-F}
f_*\Hom_{n-1}(\Sigma_{\lsub\ct \att})\subset \vo_* \Hom_{n-1}(\partial \lsub) \textnormal{ in } \Hom_{n-1}(\T_{\lsub\ct \att}\st\ct 0_{\T \st}).
\end{equation}
Since \eqref{eq:def-retr} implies \cite[Cor.~IV.16.5]{bredon1993topology} that
$$ \vt_* \Hom_{n-1}(\lsub\ct \att) = \vt_* \Hom_{n-1}(\partial \lsub) \textnormal{ in } \Hom_{n-1}(\T_{\lsub\ct \att}\st\ct 0_{\T \st}),$$ 
\eqref{eq:f-F} and \eqref{eq:def-retr-tangent} imply that \eqref{eq:no-hc-red-1} holds if 
\begin{equation}\label{eq:th-final-hom}
\vt_* \Hom_{n-1}(\partial \lsub)\not \subset \vo_* \Hom_{n-1}(\partial \lsub) \textnormal{ in } \Hom_{n-1}(\T_{\partial \lsub}\lsub \ct 0_{\T\lsub}). % \Hom_{n-1}(\T_{\partial \lsub}\st \ct 0_{\T\st})
\end{equation}
The vector field $\vo$ is nowhere zero by \eqref{eq:tf-nonzero}, $\chi(\lsub) = \chi(\att) \neq 0$ (\S \ref{sec:prelim}), and $\lsub$ is orientable since $\st$ is.
Thus, Lemma~\ref{lem:hv} establishes \eqref{eq:th-final-hom} and hence  also \eqref{eq:no-hc-red-1}. This completes the proof.
\end{proof}

\begin{remark}
The proof of Theorem~\ref{th:vb-hom-implies-adv} relies on orientability of the basin of attraction $\bas$ of $\att$ for $\vt$ (which implies orientability of $\lsub$), rather than orientability of all of $\st$.
Thus, if $\bas$ is known to be a proper subset of $\st$, one can replace the orientability assumption in Theorem~\ref{th:vb-hom-implies-adv} with a weaker assumption that implies orientability of $\bas$.
For example, it would suffice to assume that the complement of some (hence any \cite{michor1994n}) point in $\st$ is orientable.
Note that, if $\st$ is compact, then $\bas$ is always a proper subset of $\st$.
Similarly, if the \v{C}ech cohomology of $\att$ (with coefficients in any abelian group) is not isomorphic to that of $\st$, then again $\bas$ is always a proper subset of $\st$ \cite[Thm~6.3]{gobbino2001topological}. 
\end{remark}

\subsection*{Acknowledgments}
The author thanks Frederick A. Leve for useful discussions, and the organizers for inviting him to the 2023 workshop on Geometry, Topology, and Control System Design at the Banff International Research Station.
This workshop commemorated the late Roger W. Brockett and marked the 50th anniversary of the Conference at {M}ichigan {T}echnological {U}niversity \cite{michigan-tech-1983} where Brockett introduced his necessary condition that led to the results herein.

%---------------------------------------------------------------
%
% BibTeX users please use
 \bibliographystyle{unsrt}
 \bibliography{ref}
%
% Non-BibTeX users please follow the syntax
% the syntax of "referenc.tex" for your own citations
%\input{referenc}
%%%%%%%%%%%%%%%%%%%%%%%%%%%%%%%%%%%%%%%%%%%%%%%%%%%%%%%%%%%%%%%%%%%%%%

%\begin{thebibliography}{[LTY99b]}

%\bibitem[Bal76]{Bal.1}
%Balakrishnan, A.~V.:
%Applied Functional Analysis.
%Springer-Verlag, New York, (1976)

%\end{thebibliography}
%%%%%%%%%%%%%%%%%%%%%%%%%%%%%%%%%%%%%%%%%%%%%%%%%%%%%%%%%%%%%%%%%%%%%%
%
\end{document}